\documentclass[11pt]{amsart}
\usepackage{amsmath, amsthm, amsfonts, amsbsy, amssymb, upref} 
\usepackage{epsfig, graphicx}
\usepackage{pdfsync, hyperref}

\newtheorem{thm}{Theorem}
\newtheorem{lmm}[thm]{Lemma}

\newcommand{\bbx}{\mathbf{X}}

\newcommand{\bby}{\mathbf{Y}}
\newcommand{\bbz}{\mathbf{Z}}

\newcommand{\bx}{\mathbf{x}}

\newcommand{\cov}{\mathrm{Cov}}

\newcommand{\ee}{\mathbb{E}}

\newcommand{\mf}{\mathcal{F}}

\newcommand{\rr}{\mathbb{R}}

\newcommand{\ve}{\varepsilon}

\begin{document}
\title{Assumptionless consistency of the Lasso}
\author{Sourav Chatterjee}
\address{Department of Statistics, Stanford University}
\email{souravc@stanford.edu}
\thanks{Sourav Chatterjee's research was partially supported by the  NSF grant DMS-1005312}

\begin{abstract}
The Lasso is a popular statistical tool invented by Robert Tibshirani for linear regression when the number of covariates is greater than or comparable to the number of observations. The purpose of this note is to highlight the simple fact (noted in a number of earlier papers in various guises) that for the loss function considered in Tibshirani's original paper,  the Lasso is consistent under almost no  assumptions at~all. 
\end{abstract}
\maketitle

\section{Introduction}\label{intro}
The Lasso is a penalized regression procedure introduced by  Tibshirani~\cite{tibs96} in 1996. Given response variables $y_1,\ldots, y_n$ and $p$-dimensional  covariates $\bx_1,\ldots, \bx_n$, the Lasso fits the linear regression model 
\[
\ee(y_i\mid \bx_i) = {\boldsymbol \beta}\cdot \bx_i
\]
by minimizing the $\ell^1$ penalized squared error
\[
\sum_{i=1}^n (y_i - {\boldsymbol \beta} \cdot \bx_i)^2 + \lambda\sum_{j=1}^p |\beta_j|, 
\]
where ${\boldsymbol \beta} = (\beta_1,\ldots, \beta_p)$ is the vector of regression parameters and $\lambda$ is a penalization parameter. As $\lambda$ increases, the Lasso estimates are shrunk towards zero. An interesting and useful feature of the Lasso is that it is well-defined even if $p$ is greater than $n$. Not only that, often only a small fraction of the estimated $\beta_i$'s turn out to be non-zero, thereby producing an effect of automatic variable selection. And thirdly, there is a fast and simple procedure for computing the Lasso estimates  simultaneously for all $\lambda$ using the Least Angle Regression (LARS) algorithm of Efron et.\ al.\ \cite{efronetal04}. The success of Lasso stems from all of these factors. 

There have been numerous efforts to give conditions under which the Lasso `works'. Much of this work has its origins in the investigations of $\ell^1$ penalization by David Donoho and coauthors (some of it predating Tibshirani's original paper) \cite{donoho04, donohoelad02, donohohuo02, donohojohnstone94, donohoetal95}. Major advances were made by Osborne et.~al.~\cite{osborneetal00}, Knight and Fu \cite{knightfu00}, Fan and Li \cite{fanli01}, Meinshausen and B\"uhlmann \cite{meinshausenbuhlmann06},  Yuan and Lin \cite{yuanlin07}, Zhao and Yu \cite{zhaoyu07}, Zou \cite{zou06}, Greenshtein and Ritov \cite{gr}, Bunea et.\ al.~\cite{btw06, btw07}, Cand\`es and Tao \cite{candestao05, candestao07}, Zhang and Huang \cite{zhanghuang08}, Lounici \cite{lounici08}, Bickel et.\ al.~\cite{bickeletal09}, Zhang \cite{zhang09}, Koltchinskii \cite{kol09}, Rigollet and Tsybakov \cite{rigollet}, Wainwright \cite{wainwright09}, Bartlett et.~al.~\cite{bartlettetal12}, and many other authors. Indeed, it is a daunting task to compile a thorough review of the literature. Fortunately, this daunting task has been accomplished in the recent  book of B\"uhlmann and van de Geer~\cite{bg11}, to which we refer the reader for an extensive bibliography and a comprehensive treatment of the Lasso and its many variants. 


Quite recently, researchers have realized that if the performance of the Lasso is measured using the so-called `prediction loss', then almost no assumptions are required for consistency. For example, this has been noted in Rigollet and Tsybakov \cite{rigollet},  B\"uhlmann and van de Geer~\cite{bg11} and Bartlett et.~al.~\cite{bartlettetal12}. The purpose of this note is to record this simple but striking fact, and give a clean presentation with a short proof.

\vskip.3in

\section{The setup}\label{setup}
Suppose that $X_1,\ldots, X_p$ are (possibly dependent) random variables, and $M$ is a constant such that 
\begin{equation}\label{range}
|X_j|\le M
\end{equation}
almost surely for each $j$. Let
\begin{equation}\label{model}
Y = \sum_{j=1}^p\beta_j^* X_j + \ve,
\end{equation}
where $\ve$ is independent of the $X_j$'s and 
\begin{equation}\label{error}
\ve \sim N(0,\sigma^2).
\end{equation}
Here $\beta_1^*,\ldots, \beta_p^*$ and $\sigma^2$ are unknown constants. 

Let $\bbz$ denote the random vector $(Y, X_1,\ldots, X_p)$. Let $\bbz_1,\ldots, \bbz_n$ be i.i.d.\ copies of~$\bbz$. We will write $\bbz_i = (Y_i, X_{i,1},\ldots, X_{i,p})$.  The set of vectors $\bbz_1,\ldots,\bbz_n$ is our data.  The conditions \eqref{range}, \eqref{model}, \eqref{error} and the independence of $\bbz_1,\ldots,\bbz_n$ are all that we need to assume in this paper, besides the sparsity condition that $\sum_{j=1}^n |\beta_j^*|$ is not too large. 

\vskip.3in

\section{Prediction error}\label{spe}
Suppose that in the vector $\bbz$, the value of $Y$ is unknown and our task is to predict $Y$ using the values of $X_{1},\ldots, X_{p}$. If the parameter vector ${\boldsymbol \beta}^* = (\beta_1^*,\ldots, \beta_p^*)$ was known, then best predictor of $Y$ based on $X_{1},\ldots, X_{p}$ would be the linear combination 
\[
\hat{Y} := \sum_{j=1}^p \beta_{j}^* X_{j}. 
\]
However $\beta_1^*,\ldots, \beta_p^*$ are unknown, and so we need to estimate them from the data $\bbz_1,\ldots, \bbz_n$. The `mean squared prediction error' of any estimator $\tilde{\boldsymbol \beta} = (\tilde{\beta}_1, \ldots, \tilde{\beta}_p)$ is defined  as the expected squared error in estimating $\hat{Y}$ using $\tilde{\boldsymbol \beta}$, that~is, 
\begin{equation}\label{spedef}
\mathrm{MSPE}(\tilde{\boldsymbol \beta}) := \ee(\hat{Y} - \tilde{Y})^2,
\end{equation}
where
\[
\tilde{Y} := \sum_{j=1}^p \tilde{\beta}_j X_{j}. 
\]
Note that here $\tilde{\beta}_1,\ldots,\tilde{\beta}_p$ are computed using the data $\bbz_1,\ldots,\bbz_n$, and are therefore independent of $X_1,\ldots,X_p$. By this observation it is easy to see that the prediction error may be alternatively expressed as follows. Let $\Sigma$ be the covariance matrix of $(X_1,\ldots, X_p)$, and let $\|\cdot \|_\Sigma$ be the norm (or seminorm) on $\rr^p$ induced by $\Sigma$, that is, 
\[
\|\bx\|_\Sigma^2 = \bx\cdot \Sigma \bx. 
\]
With this definition, the mean squared prediction error of any estimator $\tilde{\boldsymbol \beta}$ may be written~as 
\[
\mathrm{MSPE}(\tilde{\boldsymbol \beta}) = \ee\|{\boldsymbol \beta}^* - \tilde{\boldsymbol \beta}\|_\Sigma^2. 
\]
While this alternative representation of the mean squared prediction error may make it more convenient to connect it to, say, the $\ell^2$ loss, the original definition \eqref{spedef} is more easily interpretable and acceptable from a practical point of view. 

As mentioned before, the mean squared prediction error was the measure of error considered by Tibshirani in his original paper \cite{tibs96} and also previously by Breiman \cite{breiman95} in the paper that served as the main inspiration for the invention of the Lasso (see \cite{tibs96}). Although this gives reasonable justification for proving theorems about the prediction error of the Lasso, this measure of error is certainly not the last word in judging the effectiveness of a regression procedure. Indeed, as Tibshirani \cite{tibs96} remarks, ``There are two reasons why the data analyst is often not satisfied with the OLS [Ordinary Least Squares] estimates. The first is {\it prediction accuracy} .... The second is {\it interpretation.}'' Proving that the Lasso has a small prediction error will take care of the first concern, but not  the second.

\vskip.3in
\section{Prediction consistency of the Lasso}
Take any $K > 0$ and define the estimator $\tilde{\boldsymbol \beta}^K = (\tilde{\beta}_1^K,\ldots, \tilde{\beta}_p^K)$ as the minimizer of 
\[
\sum_{i=1}^n(Y_i - \beta_1X_{i,1}-\cdots - \beta_p X_{i,p})^2
\]
subject to the constraint
\[
\sum_{j=1}^p |\beta_j| \le K. 
\]
If there are multiple minimizers, choose one according to some predefined rule. While this definition of the Lasso is not the same as the one given in Section~\ref{intro},  this is in fact the original formulation introduced by Tibshirani in  \cite{tibs96}. The two definitions may be shown to equivalent under a simple correspondence between $K$ and $\lambda$, although the correspondence involves some participation of the data. 

The following theorem shows that the Lasso estimator defined above is `prediction consistent' if $K$ is correctly chosen and $n\gg \log p$. This is the main result of this paper. 
\begin{thm}\label{lassothm}
Consider the setup defined in Section \ref{setup}. Let $K$ be any constant such that 
\begin{equation}\label{sumcond}
\sum_{j=1}^p |\beta_j^*|\le K.
\end{equation} 
Let $\mathrm{MSPE}$ stand for the mean squared prediction error, defined in Section~\ref{spe}. If $\tilde{\boldsymbol \beta}^K$ is the Lasso estimator defined above, then 
\[
\mathrm{MSPE}(\tilde{\boldsymbol \beta}^K) \le 2KM\sigma\sqrt{\frac{2\log (2p)}{n}} + 8K^2M^2\sqrt{\frac{2\log (2p^2)}{n}}. 
\]
\end{thm}
Close cousins of Theorem \ref{lassothm} have appeared very recently in the literature. The three closest results are possibly Theorem 4.1 of Rigollet and Tsybakov \cite{rigollet}, Corollary 6.1 of B\"uhlmann  and van de Geer \cite{bg11} and Theorem~1.2 of Bartlett et.~al.~\cite{bartlettetal12}. The main difference between these results and Theorem \ref{lassothm} is that the results in \cite{rigollet, bg11, bartlettetal12} are for what I call the `estimated prediction error' (discussed in the next section) whereas Theorem \ref{lassothm} is for the prediction error as defined in Tibshirani \cite{tibs96}. 
See also Foygel and Srebro~\cite{foygelsrebro11} and Massart and Meynet \cite{massartmeynet11} for some other related results.

Theorem \ref{lassothm} may be used to get error bounds for other loss functions under additional assumptions. For example, if we assume that the smallest eigenvalue of $\Sigma$ is bounded below by some number $\lambda$, then the inequality 
\[
\|\tilde{\boldsymbol \beta}^K - {\boldsymbol \beta}^*\|^2 \le \lambda^{-1} \|\tilde{\boldsymbol \beta}^K - {\boldsymbol \beta}^*\|^2_\Sigma,
\]
together with Theorem \ref{lassothm} gives a bound on the $\ell^2$ error. Similarly, assuming that ${\boldsymbol \beta}^*$ has only a small number of nonzero entries may allow us to derive stronger conclusions from Theorem \ref{lassothm}.

\vskip.3in

\section{Estimated prediction error}
Instead of the prediction error defined in Section \ref{spe}, one may alternatively consider the `estimated mean squared prediction error' of an estimator $\tilde{\boldsymbol \beta}$, defined as
\[
\widehat{\mathrm{MSPE}}(\tilde{\boldsymbol\beta}) := \frac{1}{n}\sum_{i=1}^n (\hat{Y}_i - \tilde{Y}_i)^2,
\]
where 
\[
\hat{Y}_i = \sum_{j=1}^p \beta_j^* X_{i,j} \ \text{ and } \  \tilde{Y}_i = \sum_{j=1}^p \tilde{\beta}_jX_{i,j}.
\]
Alternatively, this may be expressed as
\[
\widehat{\mathrm{MSPE}}(\tilde{\boldsymbol\beta}) = \|\tilde{\boldsymbol \beta} - {\boldsymbol \beta}^*\|_{\hat{\Sigma}}^2
\]
where
\[
\|\bx\|_{\hat{\Sigma}}^2 = \bx \cdot \hat{\Sigma} \bx,
\]
and $\hat{\Sigma}$ is the sample covariance matrix of the covariates, that is, the matrix whose $(j,k)$the element is
\[
\frac{1}{n}\sum_{i=1}^n X_{i,j} X_{i,k}. 
\]
The slight advantage of working with the estimated mean squared prediction error over the actual mean squared prediction error is that consistency in the estimated error holds if $K$ grows slower than $(n/\log p)^{1/2}$, rather than $(n/\log p)^{1/4}$ as demanded by the mean squared prediction error. This is made precise in the following theorem. 
\begin{thm}\label{msethm}
Let all notation be as in Theorem \ref{lassothm}, and suppose that \eqref{sumcond} holds. Let $\widehat{\mathrm{MSPE}}$ denote the estimated mean squared prediction error, as defined above. Then 
\[
\ee\bigl(\widehat{\mathrm{MSPE}}(\tilde{\boldsymbol\beta}^K)\bigr) \le 2KM\sigma\sqrt{\frac{2\log (2p)}{n}}. 
\]
\end{thm}
Incidentally, the above theorem is related to the notion of persistence defined in \cite{gr} and thoroughly investigated in \cite{bartlettetal12}. Theorem 4.1 of \cite{rigollet}, Corollary 6.1 of \cite{bg11} and Theorem 3.1 of \cite{massartmeynet11} are other closely related results.

\section{Proofs of Theorems \ref{lassothm} and \ref{msethm}}
Let $\bby := (Y_1,\ldots, Y_n)$, and $\tilde{\bby}^K := (\tilde{Y}_1^K,\ldots, \tilde{Y}_n^K)$, where
\[
\tilde{Y}_i^K := \sum_{j=1}^p \tilde{\beta}_j^K X_{i,j}. 
\]
Similarly, let
\[
\tilde{Y}^K := \sum_{j=1}^p \tilde{\beta}_j^K X_{j}. 
\]
For each $1\le j\le p$, let $\bbx_j := (X_{1,j},\ldots, X_{n,j})$. Finally, let $\hat{\bby} := (\hat{Y}_1,\ldots, \hat{Y}_n)$, where 
\[
\hat{Y}_i := \sum_{j=1}^p \beta_j^* X_{i,j},
\]
Given $\bbz_1,\ldots, \bbz_n$, define the  set 
\[
C := \{\beta_1\bbx_1+\cdots + \beta_p \bbx_p: |\beta_1|+\cdots+|\beta_p|\le K\}. 
\]
Note that $C$ is a compact convex subset of $\rr^n$. By definition, $\tilde{\bby}^K$ is the projection of $\bby$ on to the set  $C$. Since $C$ is convex, it follows that for any $\bx\in C$, the vector $\bx - \tilde{\bby}^K$ must be at an obtuse angle to the vector $\bby - \tilde{\bby}^K$. That is, 
\[
(\bx - \tilde{\bby}^K)\cdot  (\bby - \tilde{\bby}^K)\le 0. 
\]
The condition \eqref{sumcond} ensures that $\hat{\bby} \in C$. Therefore
\[
(\hat{\bby} - \tilde{\bby}^K)\cdot  (\bby - \tilde{\bby}^K)\le 0. 
\]
This may be written as
\begin{align*}
\|\hat{\bby} - \tilde{\bby}^K\|^2 &\le (\bby - \hat{\bby})\cdot (\tilde{\bby}^K - \hat{\bby})\\
&= \sum_{i=1}^n \ve_i\biggl(\sum_{j=1}^p (\tilde{\beta}_j^K - \beta_j^*)X_{i,j}\biggr)\\
&= \sum_{j=1}^p (\tilde{\beta}_j^K - \beta_j^*)\biggl(\sum_{i=1}^n \ve_i X_{i,j}\biggr).   
\end{align*}
By the condition \eqref{sumcond} and the definition of $\tilde{\boldsymbol \beta}^K$, the above inequality implies that 
\begin{align}\label{est1}
\|\hat{\bby} - \tilde{\bby}^K\|^2 \le 2K \max_{1\le j\le p} |U_j|,
\end{align}
where
\[
U_j := \sum_{i=1}^n \ve_i X_{i,j}. 
\]
Let $\mf$ be the sigma algebra generated by $(X_{i,j})_{1\le i\le n, \, 1\le j\le p}$. Let $\ee^\mf$ denote the conditional expectation given $\mf$. Conditional on~$\mf$, 
\[
U_j \sim N\biggl(0,\, \sigma^2 \sum_{i=1}^n X_{i,j}^2\biggr). 
\]
Since $|X_{i,j}|\le M$ almost surely for all $i,j$, it follows from the standard results about Gaussian random variables (see Lemma \ref{tail1} in the Appendix) that 
\[
\ee^\mf(\max_{1\le j\le p} |U_j|)\le M\sigma\sqrt{2n \log (2p)}. 
\]
Since the right hand side is non-random, it follows that
\[
\ee(\max_{1\le j\le p} |U_j|)\le M\sigma\sqrt{2n \log (2p)}. 
\]
Using this bound in \eqref{est1}, we get
\begin{align}\label{est2}
\ee\|\hat{\bby} - \tilde{\bby}^K\|^2 \le 2KM\sigma\sqrt{2n \log (2p)}. 
\end{align}
This completes the proof of Theorem \ref{msethm}. For Theorem \ref{lassothm}, we have to work a bit more. Note that by the independence of $\bbz$ and $\tilde{\boldsymbol \beta}^K$, 
\begin{align*}
\ee^\mf(\hat{Y}-\tilde{Y}^K)^2 &= \sum_{j,k=1}^p (\beta_j^*-\tilde{\beta}_j^K)(\beta_k^*-\tilde{\beta}_k^K)\ee(X_j X_k).
\end{align*}
Also, we have
\begin{align*}
&\frac{1}{n} \|\hat{\bby} - \tilde{\bby}^K\|^2 \\
&= \frac{1}{n}\sum_{i=1}^n \sum_{j,k=1}^p (\beta_j^*-\tilde{\beta}_j^K)(\beta_k^*-\tilde{\beta}_k^K)X_{i,j} X_{i,k}. 
\end{align*}
Therefore, if we define
\[
V_{j,k} :=\ee(X_j X_k) - \frac{1}{n}\sum_{i=1}^n X_{i,j} X_{i,k},
\]
then 
\begin{align}
\ee^\mf (\hat{Y}-\tilde{Y}^K)^2 - \frac{1}{n} \|\hat{\bby} - \tilde{\bby}^K\|^2&= \sum_{j,k=1}^p (\beta_j^*-\tilde{\beta}_j^K)(\beta_k^*-\tilde{\beta}_k^K)V_{j,k}\nonumber \\
&\le 4K^2 \max_{1\le j,k\le p} |V_{j,k}|. \label{est3}
\end{align}
Since $|\ee(X_j X_k)-X_{i,j}X_{i,k}|\le 2M^2$ for all $i$, $j$ and $k$, it follows by Hoeffding's inequality (see Lemma \ref{hoeffding} in the Appendix) that for any $\beta \in \rr$,  
\[
\ee(e^{\beta V_{j,k}}) \le e^{2\beta^2M^4/n}.
\]
Consequently,  by Lemma \ref{tail2} from the Appendix, 
\[
\ee(\max_{1\le j,k\le p} |V_{j,k}|) \le 2M^2 \sqrt{\frac{2\log (2p^2)}{n}}. 
\]
Plugging this into \eqref{est3} and combining with \eqref{est2} completes the proof of Theorem \ref{lassothm}. 

\vskip.3in
\section*{Appendix}
The following inequality is a well-known result about the size of the maximum of Gaussian random variables. 
\begin{lmm}\label{tail1}
Suppose that $\xi_i\sim N(0, \sigma_i^2)$, $i=1,\ldots,m$. The $\xi_i$'s need not be independent. Let $L := \max_{1\le i\le m} \sigma_i$. Then
\[
\ee(\max_{1\le i\le m}|\xi_i|) \le L\sqrt{2\log (2m)}. 
\]
\end{lmm}
\begin{proof}
For any $\beta \in \rr$, $\ee(e^{\beta \xi_i}) =e^{\beta^2\sigma_i^2/2}\le e^{\beta^2 L^2/2}$. Thus, for any $\beta > 0$,
\begin{align*}
\ee(\max_{1\le i\le m} |\xi_i|) &= \frac{1}{\beta} \ee(\log e^{\max_{1\le i\le m} \beta |\xi_i|})\\
&\le \frac{1}{\beta} \ee\biggl(\log \sum_{i=1}^m(e^{\beta \xi_i} + e^{-\beta \xi_i})\biggr)\\
&\le \frac{1}{\beta} \log \sum_{i=1}^m\ee(e^{\beta \xi_i} + e^{-\beta \xi_i})\le \frac{\log(2m)}{\beta} + \frac{\beta L^2}{2}. 
\end{align*}
The proof is completed by choosing $\beta = L^{-1}\sqrt{2\log (2m)}$. 
\end{proof}
The result extends easily to the maximum of random variables with Gaussian tails. 
\begin{lmm}\label{tail2}
Suppose that for $i=1,\ldots, m$, $\xi_i$ is a random variable such that $\ee(e^{\beta \xi_i}) \le e^{\beta^2L^2/2}$ for each $\beta \in \rr$, where $L$ is some given constant. Then 
\[
\ee(\max_{1\le i\le m}|\xi_i|) \le L\sqrt{2\log (2m)}. 
\]
\end{lmm}
\begin{proof}
Exactly the same as the proof of Lemma \ref{tail1}. 
\end{proof}
The following lemma is commonly known as Hoeffding's inequality \cite{hoeffding63}. The version we state here is slightly different than the commonly stated version. For this reason, we state the lemma together with its proof. 
\begin{lmm}\label{hoeffding}
Suppose that $\eta_1,\ldots, \eta_m$ are independent, mean zero  random variables, and $L$ is a constant such that $|\eta_i|\le L$ almost surely for each $i$. Then for each $\beta \in \rr$, 
\[
\ee(e^{\beta\sum_{i=1}^m \eta_i}) \le e^{\beta^2mL^2/2}. 
\]
\end{lmm}
\begin{proof}
By independence,
\[
\ee(e^{\beta\sum_{i=1}^m \eta_i}) = \prod_{i=1}^m \ee(e^{\beta\eta_i}). 
\]
Therefore it suffices to prove the result for $m=1$. Note that
\[
\ee(e^{\beta \eta_1}) = \int_{-L}^L e^{\beta x} d\mu_1(x),
\]
where $\mu_1$ is the law of $\eta_1$. By the convexity of the map $x\mapsto e^{\beta x}$, it follows that for each $x\in [-L, L]$,
\begin{equation}\label{hoeff}
e^{\beta x} = e^{\beta (tL + (1-t)(-L))}\le te^{\beta L } + (1-t) e^{-\beta L},
\end{equation}
where
\[
t = t(x) = \frac{x}{2L} + \frac{1}{2}. 
\]
Since $\ee(\eta_1)=0$, therefore $\int t(x)d\mu_1(x) = 1/2$. Thus by \eqref{hoeff}, $\ee(e^{\beta \eta_1}) \le \cosh(\beta L)$. The inequality $\cosh x \le e^{x^2/2}$ completes the proof. 
\end{proof}

\vskip.2in

\noindent{\bf Acknowledgments.} The author thanks Alexandre Tsybakov, Peter Bickel, Peter B\"uhlmann, Peter Bartlett, Mathias Drton, Gilles Blanchard, Ashley Petersen and Persi Diaconis for helpful comments.

\end{document}